\pdfoutput=1 
\documentclass{amsproc}
\usepackage{latexsym}
\usepackage{amssymb}
\usepackage{amsmath}
\usepackage{amscd}
\usepackage[all]{xy}
\usepackage{graphicx}
\usepackage{float}

\newtheorem{thm}{Theorem}
\newtheorem{pro}{Proposition}
\newtheorem{lem}{Lemma}

\newtheorem{dfn}{Definition}
\newcommand{\Area}{\operatorname{Area}}
\newcommand{\Vol}{\operatorname{Vol}}
\newcommand{\Log}{\operatorname{Log}}

\newcommand{\Arg}{\operatorname{Arg}}

\font\smc=cmcsc10

\title{Discriminant coamoebas in dimension two}

\author{Lisa Nilsson \& Mikael Passare}
\address{Department of Mathematics,
Stockholm University,
\hfill\break
\phantom{\hskip.54cm} SE-106 91  Stockholm, Sweden}
\email{lisa@math.su.se, passare@math.su.se}

\date{\today}

\begin{document}
\maketitle

\begin{abstract}

This paper deals with coamoebas, that is, images under coordinate\-wise argument mappings, of certain quite particular plane algebraic curves. These curves are the zero sets of reduced $A$-discriminants of two variables. We consider the coamoeba primarily as a subset of the torus $\mathbf{T}^2=(\mathbf{R}/2\pi\mathbf{Z})^2$, but also as a subset of its covering space 
$\mathbf{R}^2$, in which case the coamoeba consists of an infinite, doubly periodic image. In fact, it turns out to be natural to take multiplicities into account, and thus to treat the coamoeba as a chain in the sense of algebraic topology.

We give a very explicit description of the coamoeba as the union of two mirror images of a (generally non-convex) polygon, which is easily constructed from a matrix $B$ that represents the Gale transform of the original collection $A$.
We also give an area formula for the coamoeba, and we show that the coamoeba is intimately related to a certain zonotope. In fact, on the torus $\mathbf{T}^2$ the coamoeba and the zonotope together form a cycle, and hence precisely cover the entire torus an integer number of times. This integer is proved to be equal to the (normalized) volume of the convex hull of $A$.

\end{abstract}

\section{Introduction}

\noindent The objects that we study in this article are special algebraic curves defined by reduced, or inhomogeneous, $A$-discriminantal polynomials; and more specifically, the coamoebas of these $A$-discriminantal curves. Such a coamoeba is the image of the curve under the simple mapping that takes each complex coordinate to its argument. The notion of $A$-discriminants was originally introduced by Gelfand, Kapranov, and Zelevinsky, and the monograph \cite{GKZ} contains a thorough account of the subject. Here we just recall a few facts that will help put our results in proper perspective. The starting point is a finite configuration $A$ of points 
$\alpha_1,\alpha_2,\ldots,\alpha_N$ in $\mathbf{Z}^n$, and the corresponding family of Laurent polynomials
$$\phantom{------} f(x)=\sum_{k=1}^Na_k\,x^{\alpha_k}\ \in \ \mathbf{C}[x_1^{\pm 1},\ldots,x_n^{\pm 1}]$$
with exponent vectors from $A$. Granted some generically satisfied conditions, see \cite{GKZ}, the set of coefficients vectors 
$(a_1,\ldots, a_N)$, for which the hypersurface $\mathcal{Z}_f=\{\,x\in\mathbf{C}_*^n\,;\  f(x)=0\,\}$ is not a smooth manifold, coincides with the zero locus of an irreducible polynomial $D_A\in\mathbf{Z}[a_1,\ldots, a_N]$.
This \emph{$A$-discriminant} $D_A$ is uniquely determined up to sign, provided its coefficients are taken to be relatively prime, and its dependence on the point configuration $A$ is such that it remains invariant under $\mathbf{Z}$-affine isomorphisms.
\smallskip

As is customary in this context, we will make a slight abuse of notation, and denote by $A$ also the associated $(1+n)\times N$-matrix
\begin{equation}\label{amatrix}
A=\begin{pmatrix}
1 & \ldots  & 1 \\
\alpha_1 & \ldots &\alpha_N\\
\end{pmatrix},\,
\end{equation}
with the $\alpha_k$ here being viewed as column vectors.  We shall make two standard assumptions on the matrix $A$, and consequently also on the corresponding point configuration:

$\bullet\hskip.5cm \text{rank}(A)=1+n$, so in particular $N>n$;

$\bullet\hskip.4cm$ the maximal minors of $A$ are relatively prime.

\noindent The latter condition means that the columns of $A$ generate the full lattice $\mathbf{Z}^{1+n}$.
\smallskip

Each row vector of the matrix $A$ corresponds to a (quasi-)homogeneity of the $A$-discriminant, which means that $D_A$ can actually be considered as
a function of only $m:=N-n-1$ essential variables. An explicit method for dehomogenizing the $A$-discriminant consists in choosing a \emph{Gale transform} of $A$.
This is simply an integer $N\times m$-matrix $B$, whose column vectors form a $\mathbf{Z}$-basis for the kernel of the linear map represented by the matrix $A$.
The row vectors $b_1, b_2,\ldots, b_N$ of $B$ then constitute a point configuration in $\mathbf{Z}^m$ which is said to be a Gale transform of the original configuration $A$, see Section 5.4 of \cite{G}. Now, the column vectors of $B$ can be used to produce new inhomogeneous coordinates for $D_A$ by the rule
$$x_k=a_1^{b_{1k}}a_2^{b_{2k}}\cdots a_N^{b_{Nk}}\,,\quad k=1,\ldots, m\,.$$
More precisely, one has a formula $D_A(a_1,\ldots,a_N)=M(a)\,D_B(x_1,\ldots,x_m)$, where $M(a)$ is a monomial in the original $a$-variables  and $D_B$
is the corresponding \emph{reduced} $A$-discriminant. It is natural to denote this reduced discriminant by $D_B$, for it is the choice of the matrix $B$ that determines it. In fact, one could equally well start the whole theory from a $B$-matrix whose row vectors sum up to zero, and then take a Gale transform $A$ of the special form (\ref{amatrix}), which will be uniquely determined up to a $\mathbf{Z}$-affine isomorphism. Note that for this to work $B$ should not contain any row with only zeros.
\bigskip

In this paper we are going to focus on the case where $N=n+3$, so that $m=2$, and  the reduced $A$-discriminant is a function of
two variables. Before proceeding, let us look at three simple cases, which we will follow throughout the paper.
\bigskip

{\smc Examples.} ({\it i$\,$})\  Consider first the trivial $A$-matrix $A=(1\ 1\ 1)$, which amounts to the case of dimension $n=0$. It is then natural to define the corresponding $A$-discriminant to be the linear form $D_A=a_1+a_2+a_3$, and a suitable $B$-matrix for this case is given by
\vskip-.4cm
$$B=\left(\begin{array}{cccc}
 \!\!\!-1 & \!\!\! -1 \\
\,\,1 & \,\,0 \\
\,\,0 & \,\,1 
\end{array}\right).$$
Writing $ x_1=a_1^{-1}a_2$ and $x_2=a_1^{-1}a_3$, we then have $D_A(a)= a_1\,D_B(x)$ with $D_B(x)=1+x_1+x_2$.
\bigskip

({\it ii$\,$})\  Now let $n=1$ and let $A$ be the configuration $\alpha_k=k-1$ with $k=1,2,3,4$. That is, we have the matrix
$$A=\left(\begin{array}{cccc}
1 & 1 & 1 & 1\\
 0 & 1 & 2 & 3
\end{array}\right). $$
The corresponding $A$-discriminant is given by
$$D_A(a)=27a_1^2a_4^2+4a_1a_3^3+4a_2^3a_4-18a_1a_2a_3a_4-a_2^2a_3^2\,,$$ 
and we recognize it as the classical cubic discriminant that vanishes precisely when the algebraic equation $a_1+a_2x+a_3x^2+a_4x^3=0$
has some multiple root. A Gale transform of $A$ is
$$B=\left(\begin{array}{cccc}
\,\,1 & \,\,0 \\
\,\,0 & \,\,1 \\
 \!\!\!-3 & \!\!\! -2 \\
\,\,2 & \,\,1 
\end{array}\right),$$
and one has $D_A(a)=a_3^6a_4^{-2}\,D_B(x)$ with $x_1=a_1a_3^{-3}a_4^2$, $x_2=a_2a_3^{-2}a_4$, so the reduced $A$-discriminant is given by $D_B(x)=27x_1^2+4x_1+4x_2^3-18x_1x_2-x_2^2$.
\bigskip

({\it iii$\,$})\  As a third example we take the two-dimensional configuration $\alpha_1=(2,3)$, $\alpha_2=(1,2)$, $\alpha_3=(0,0)$, $\alpha_4=(2,1)$, and $\alpha_5=(3,2)$.
In this case one can choose the Gale transform to be
$$B=\left(\begin{array}{cccc}
\,\,1 &  \!\!\!-1 \\
 \!\!\!-1 & \,\,2 \\
 \,\,0 &  \!\!\!-2 \\
 \,\,1 & \,\,3 \\
 \!\!\!-1 & \!\!\! -2 
\end{array}\right).$$
The corresponding inhomogeneous coordinates are this time $x_1=a_1a_2^{-1}a_4a_5^{-1}$ and $x_2=a_1^{-1}a_2^2a_3^{-2}a_4^3a_5^{-2}$, and
the $A$-discriminant takes the form
$$D_A(a)=a_1a_2^3a_3^4a_4^{-3}a_5^7\,D_B(x)\,,$$
with the reduced $A$-discriminant $D_B$ given by
$$D_B(x)=1024x_2-1280x_1x_2-3125x_1^3+40x_1^2x_2+432x_1x_2^2+40x_1^3x_2\phantom{----}$$
\vskip-.6cm $$\phantom{-------------}-864x_1^2x_2^2-1280x_1^4x_2+432x_1^3x_2^2+1024x_1^5x_2\,.$$
\bigskip

The aim of this article is to describe the so-called coamoeba $\mathcal{A}_B'$ of an arbitrary reduced $A$-discriminantal curve $D_B(x_1,x_2)=0$. Let us therefore first define what is meant by the coamoeba of a general algebraic variety in the complex torus $\mathbf{C}_*^n$.
\bigskip
\begin{dfn}\label{coamdef}
{\rm Let $\mathcal{Z}\subset \mathbf{C}_*^n$ be an algebraic variety. The \emph{coamoeba} $\mathcal{A}'_{\mathcal{Z}}$ of $\mathcal{Z}$ is defined to be the subset of the real torus $\mathbf{T}^n=(\mathbf{R}/2\pi\mathbf{Z})^n$ given by $\mathcal{A}'_{\mathcal{Z}}=\Arg(\mathcal{Z})$, where $\Arg$ denotes the mapping
$$(x_1,x_2,\ldots,x_n)\mapsto\bigl(\arg(x_1),\arg(x_2),\ldots,\arg(x_n)\bigr).$$
}\end{dfn}

{\it Remark.} The coamoeba was originally introduced as a kind of dual object of the \emph{amoeba} $\mathcal{A}_{\mathcal{Z}}$ of $\mathcal{Z}$, which is analogously defined as 
$\mathcal{A}_{\mathcal{Z}}=\Log(\mathcal{Z})$, with $\Log(x)\!=\!\bigl(\log|x_1|,\ldots,\log|x_n|\bigr)$.  Some authors use the synonymous term alga instead of coamoeba. It is useful to think of the coamoeba also as ``lifted" to the covering space $\mathbf{R}^n$, thus regarding the argument mapping as being multivalued. 
\bigskip

A very efficient way of studying the zero locus of a reduced $A$-discriminant is to make use of the so-called Horn--Kapranov parametrization. In order to explain it, we first recall what is meant by the logarithmic Gauss mapping $\gamma$ defined on the regular part of complex hypersurface $\mathcal{Z}\subset\mathbf{C}_*^m$. In terms of a defining function $f$ for $\mathcal{Z}$ the logarithmic Gauss mapping is given analytically by the simple formula 
$$\mathrm{reg}\,\mathcal{Z}\owns x\mapsto\gamma(x)=[x_1\partial_1f(x):x_2\partial_2f(x):\ldots:x_m\partial_mf(x)]\in\mathbf{C}P_{m-1},\  \partial_k=\partial/\partial x_k,$$ 
and it admits the following geometric interpretation. Given a point  $x_0\in\mathrm{reg}\,\mathcal{Z}$, choose a local branch near $x_0$ of the logarithmic transformation 
$(x_1,\ldots,x_n)\mapsto (\log x_1,\ldots,\log x_n)$. Then $\gamma(x_0)\in\mathbf{C}P_{m-1}$ is equal to the complex normal direction of the transformed surface $\mathcal{Z}'$ at the corresponding point $x_0'$. It was proved by Kapranov in \cite{K} that if the hypersurface $\mathcal{Z}$ is defined by a reduced $A$-discriminant $D_B$, then the logarithmic Gauss mapping is birational, with a completely explicit inverse rational mapping $\Psi$.
\smallskip

\begin{dfn} {\rm The \emph{Horn--Kapranov parametrization} of the discriminant hyper\-surface $D_B(x)=0$ is the rational mapping $\Psi\colon\mathbf{C}P_{m-1}\to\mathbf{C}^m_*$ given by 
\begin{equation}\label{horn}
 \Psi[t_1:t_2\ldots:t_m]=\Big(\prod_{j=1}^{N}\langle b_j,t\rangle^{b_{j1}},\prod_{j=1}^{N}\langle b_j,t\rangle^{b_{j2}},\ldots,\prod_{j=1}^{N}\langle b_j,t\rangle^{b_{jm}}\Big)\,.\end{equation}
}
\end{dfn}

{\it Remark.} The most striking result in the Kapranov paper \cite{K} is that the property of having a birational logarithmic Gauss mapping actually characterizes the reduced $A$-discriminantal hypersurfaces. The precise content of this characterization requires some clarification, and we refer to \cite{CD} for a discussion on this matter.
\bigskip

{\smc Examples.} ({\it i$\,$})\  With $B$ given by the row vectors $b_1=(-1,-1)$, $b_2=(1,0)$, and $b_3=(0,1)$, one obtains the Horn--Kapranov mapping
$$\Psi[1:t]=\Bigl(-\frac{1}{1+t},-\frac{t}{1+t}\Bigr)\,,$$
and putting $x_1=-1/(1+t)$, $x_2=-t/(1+t)$ one sees that it does indeed parametrize the line $1+x_1+x_2=0$.
\bigskip

({\it ii$\,$})\  Taking $N=4$ and the matrix $B$ with row vectors $b_1=(1,0)$, $b_2=(0,1)$, $b_3=(-3,-2)$, and $b_4=(2,1)$, one constructs the parametrization
$$x_1=-\frac{(2+t)^2}{(3+2t)^3}\,,\quad x_2=\frac{t(2+t)}{(3+2t)^2}$$
of the reduced cubic discriminant curve $\,\,27x_1^2+4x_1+4x_2^3-18x_1x_2-x_2^2=0$.
\bigskip

({\it iii$\,$})\  Our example with $N=5$ is the $B$-matrix with row vectors $b_1=(1,-1)$, $b_2=(-1,2)$, $b_3=(0,-2)$, $b_4=(1,3)$, and $b_5=(-1,-2)$. The corresponding Horn--Kapranov parametrization is then
\begin{equation}\label{tre}
x_1=-\frac{(1-t)(1+3t)}{(-1+2t)(1+2t)}\,,\quad x_2=\frac{(-1+2t)^2(1+3t)^3}{(1+2t)^2(1-t)4t^2}\,,
\end{equation}
and one can verify that it does indeed satisfy $D_B(x_1,x_2)=0$, where $D_B$ is the rather complicated reduced $A$-discriminant we found in case ({\it iii$\,$}) above.
\bigskip

We have now set the stage for our investigations of the coamoebas of reduced $A$-discriminant hypersurfaces in two variables, or in other words, the coamoebas of plane algebraic curves of the special form $D_B(x_1,x_2)=0$, with the irreducible polynomial $D_B\in\mathbf{Z}[x_1,x_2]$ derived from an integer $N\times2$-matrix $B$ as explained above. 
\smallskip

Let us finish this introduction with some comments on the possible applications of our results on discriminant coamoebas. It was the study of the singularities of $A$-hypergeometric functions that motivated Gelfand, Kapranov, and Zelevinsky to introduce the $A$-discriminants and to produce the already classical monograph \cite{GKZ}. These $A$-hypergeometric functions appear as solutions to certain systems of differential equations that are naturally associated with a given point configuration $A$. It is thus not unexpected that coamoebas of $A$-discriminant curves should be related to the $A$-hypergeometric functions. In fact, the complement components of the discriminant coamoeba provide the exact convergence domains for certain natural Mellin-Barnes integral representations of the corresponding $A$-hypergeometric functions, see \cite{L}. Analogous results regarding discriminant amoebas and convergence domains for $A$-hypergeometric series were obtained in \cite{L} and \cite{PTs}. Another, and maybe more surprising, application of discriminant coamoebas is their appearance in modern theoretical physics. It turns out that exactly the type of coamoebas that we study here also play a role in certain linear sigma models occurring in string theory, see for instance the two-parameter models presented in Section~4.4 of \cite{HHP}.
\smallskip

Part of this work was carried out while both authors were participating in the MSRI program on Tropical Geometry in the fall of 2009. 
\bigskip

\section{Description of the coamoeba}

\noindent Let there be given an integer matrix $B$ of size $N\times 2$. It is assumed that no row vector of $B$ is equal to $(0,0)$, and also that the $2\times2$-minors of $B$ are relatively prime. As explained in the introduction, the row vectors of the matrix $B$ can be seen as representing a Gale transform of a point configuration $A=\{\alpha_1,\ldots,\alpha_N\}$ in $\mathbf{Z}^n$, with $n=N-3$. In particular, the corresponding matrix $A$, of the form (\ref{amatrix}), satisfies $AB=0$. Even though the configuration $A$ is not uniquely determined by $B$, it can only change by a $\mathbf{Z}$-affine isomorphism, so the volume of the convex hull $\text{conv}(A)$ is well defined. We shall denote by $d_B$ the normalized volume $n!\text{Vol}(\text{conv}(A))$. It is a positive integer.
\smallskip

In view of the Horn--Kapranov parametrization (\ref{horn}), for the case $m=2$, the coamoeba $\mathcal{A}_B'$ of the reduced $A$-discriminant curve $D_B(x_1,x_2)=0$ can be identified with the full image of the composed mapping $\Arg\circ\Psi\colon\mathbf{C}P_1\to\mathbf{T}^2=(\mathbf{R}/2\pi\mathbf{Z})^2$, given explicitly by
$$[1:t]\longmapsto \bigl(\,\sum_{j=1}^{N}b_{j1}\arg(b_{j1}+b_{j2}t)\,,\ \sum_{j=1}^{N}b_{j2}\arg(b_{j1}+b_{j2}t)\,\bigr)\,.$$
\smallskip

The mapping $\Arg\circ\Psi$ is generically finite, and we want to keep track of the number of preimages of various points of the torus $\mathbf{T}^2$. To this end, we shall consider the coamoeba $\mathcal{A}_B'$ as a (simplicial) chain in the sense of algebraic topology. Strictly speaking, this means that the underlying set of the coamoeba chain, which we still denote by $\mathcal{A}_B'$, is going to be the closure of the actual coamoeba. For our purposes this is not an important issue, but we briefly comment on the set-theoretical boundary of the coamoeba in a remark at the end of this section.
\smallskip

Our approach to the coamoeba chain $\mathcal{A}_B'$ is to specify a particular branch of the lifted mapping $\Arg\circ\Psi\colon\mathbf{C}P_1\to\mathbf{R}^2$, thus obtaining what we call the principal coamoeba chain $\mathcal{A}_0'\subset\mathbf{R}^2$, which is then mapped onto $\mathcal{A}_B'$ by the natural projection $\mathbf{R}^2\to\mathbf{T}^2$. We begin by describing a simplicial $1$-cycle $\Gamma_0\subset\mathbf{R}^2$, that is, an oriented polygonal curve, which will turn out to be the boundary of the principal coamoeba chain $\mathcal{A}_0'$. It will be practical for us to re-order the rows of our matrix $B$, if necessary, so as to have row vectors $b_j$ with decreasing normal slopes $\beta_j= -b_{j1}/b_{j2}$ as follows: 
$$\infty>\beta_1\ge\beta_2\ge\ldots\ge\beta_N\ge(-)\infty\,.$$
\vskip.15cm
\noindent The $1$-cycle $\Gamma_0$ is going to be the sum of two cycles $\Gamma_+$ and $\Gamma_-$, which are located symmetrically with respect to their common starting point $p_0$.
This point $p_0$ is one of the four points $(0,0)$, $(0,\pi)$, $(\pi,0)$, or $(\pi,\pi)$, depending on the signs of the components of $\Psi[1:t]$ for large real values of $t$.
\smallskip

\begin{dfn}\label{argo}
{\rm Let $[\beta_N,\beta_1]$ be the real (possibly infinite) line segment that also contains the remaining points $\beta_{N-1},\ldots,\beta_2$. The \emph{principal branch} $[\Arg\circ\Psi]_0$ on the simply connected surface $\mathbf{C}P_1\setminus[\beta_N,\beta_1]$ is defined to be the continous branch of $\Arg\circ\Psi$ that takes one of the values $(0,0)$, 
$(0,\pi)$, $(\pi,0)$, or $(\pi,\pi)$ for real $t>\beta_1$.
}
\end{dfn}

The construction of the $1$-cycle $\Gamma_+$ now goes as follows: place the dilated first vector $\pi b_1$ at the starting point $p_0=[\Arg\circ\Psi]_0(t)$, $t>\beta_1$; then place the next vector $\pi b_2$ at the new point $p_0+\pi b_1$, and continue in this fashion so that the last vector closes the cycle at the point $p_0+\pi b_1+\pi b_2+\ldots+\pi b_N=p_0$. If there are some vectors, say $b_k, b_{k+1},\ldots,b_{k+\ell}$, that are parallel, then they should be replaced by their sum $b_k+b_{k+1}+\ldots+b_{k+\ell}$ during this process. This will clearly not change the homology of the resulting $1$-cycle $\Gamma_+$. The cycle $\Gamma_-$ is then constructed in complete analogy, by starting at $p_0$ and placing the opposite vectors $-\pi b_1,-\pi b_2,\ldots,-\pi b_N$ one after another. Finally, we put $\Gamma_0=\Gamma_++\Gamma_-$.
\bigskip

{\smc Examples.} ({\it i$\,$})\  In our first example case we should re-order the vectors to get $b_1=(0,1)$, $b_2=(-1,-1)$, and $b_3=(1,0)$. Then we obtain $\beta_1=0$, $\beta_2=-1$, and $\beta_3=\infty$. Since we have $\Psi[1:t]=\bigl(-1/(1+t),-t/(1+t)\bigr)$ with both entries being negative for real $t>0$, we find that the starting point $p_0$ for the cycles $\Gamma_\pm$ is given by $(\pi,\pi)$. The triangular shaped cycle $\Gamma_+$ is depicted on the left in Figure~1.
\bigskip

({\it ii$\,$})\  The prescribed new ordering of the vectors $b_k$ on our second example is $b_1=(0,1)$, $b_2=(-3,-2)$, $b_3=(2,1)$, and $b_4=(1,0)$. Recalling the Horn--Kapranov mapping $\Psi$ for this case, we see that it has its first component negative, and the second component positive, for $t>0=\beta_1$. This time the starting point $p_0$ is thus given by $(\pi,0)$, as indicated in the lower center of  Figure~1.
\bigskip

({\it iii$\,$})\  The row vectors $b_k$ from our third example turn out to already have the correct ordering, with normal slopes $\beta_1=1$, $\beta_2=1/2$, $\beta_3=0$, $\beta_4= -1/3$, and $\beta_5=-1/2$. Plugging in a real value $t>1$ in the parametrization (\ref{tre}), we see that $x_1$ becomes positive, and $x_2$ negative. Therefore we have $p_0=(0,\pi)$ in this case. On the right in Figure~1 we have drawn the cycle $\Gamma_+$ for this third example.

\begin{figure}[H]\label{fig1}
\begin{center}
\vskip-.4cm
\includegraphics[width=10cm]{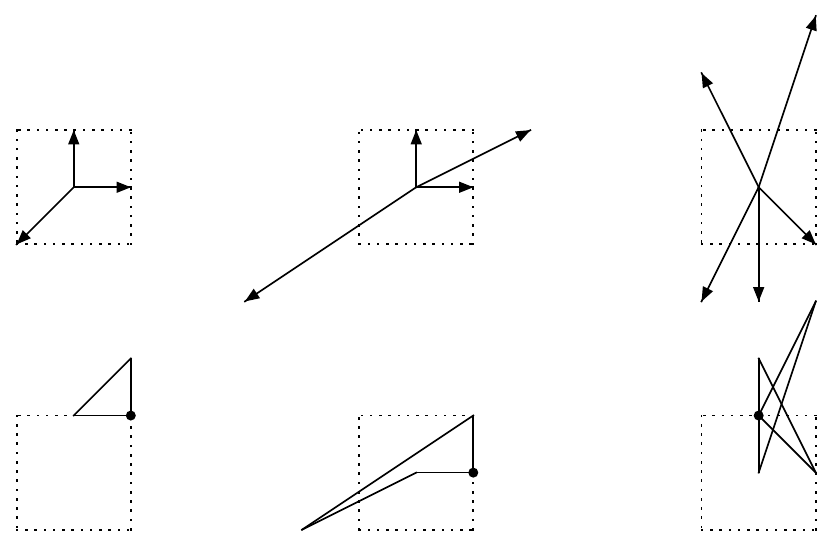}
\end{center}
\vskip.2cm
\caption{\it The vectors $\pi b_k$ (upper row) and the cycle $\Gamma_+$ with the starting point $p_0$ indicated (lower row) from our three cases. The dashed square is centered at the origin and has side length $2\pi$.}
\end{figure}

\begin{dfn}
{\rm
 The \emph{principal coamoeba chain} is given by $\mathcal{A}'_0=\mathcal{A}'_++\mathcal{A}'_-$ where $\mathcal{A}'_+$ and $\mathcal{A}'_-$ denote the images (counting multiplicities) under the principal branch mapping $[\Arg\circ\Psi]_0$ of the upper and lower half planes $\{\mathrm{Im}\, t> 0\}$ and $\{\mathrm{Im}\, t<0\}$ respectively.
}
\end{dfn}

{\it Remark.} Note that if we map the principal coamoeba chain $\mathcal{A}'_0$ to the torus by the canonical projection $ \mathbf{R}^2\to\mathbf{T}^2$ we recover the full coamoeba chain $\mathcal{A}'_B$. 

\begin{thm}\label{sats1}
The principal coamoeba chain $\mathcal{A}'_0$ has the boundary $\partial\mathcal{A}'_0=\Gamma_0$. More precisely, for each point $p\in\mathbf{R}^2\setminus\Gamma_0$, the winding number of $\Gamma_0$ with respect to $p$ is equal to the number of preimages of $p$ under the principal branch mapping $[\Arg\circ\Psi]_0$.
\end{thm}

\begin{proof}
It will be enough to prove the second statement in the theorem, for it implies the first one. We therefore fix a point $p$ in $\mathbf{R}^2\setminus\Gamma_0$. Let us make the temporary assumption that no two vectors $b_k$ are parallel, so that the values $\beta_1, \beta_2,\ldots , \beta_N$ are all different. We deal first with the image $\mathcal{A}_+'$ of the upper half plane and with the corresponding cycle $\Gamma_+$. For positive constants $\delta$ and $R$, with $\delta$ small and $R$ large, we shall consider the oriented curve $C_{\delta, R}$ in the closed upper half plane 
$\{\text{\rm Im}\,t\ge0\}$ which starts at the point $R$, then runs leftwards along the real axis, with small circular indentations of radius $\delta$ centered at the points $\beta_k$, to the point $-R$, and finally follows the half circle of radius $R$ back to the starting point. We denote by $\Gamma_{\delta, R}$ the curve in $\mathbf{R}^2$ which is the image of $C_{\delta, R}$ under the principal branch mapping $[\text{Arg}\circ\Psi]_0$.\smallskip

We shall now prove that, for all sufficiently small $\delta$ and sufficiently large $R$, the point $p$ is contained in the complement $\mathbf{R}^2\setminus\Gamma_{\delta, R}$, and  the winding numbers of the two curves  $\Gamma_{\delta, R}$ and $\Gamma_+$ with respect to $p$ are the same. First we note that, since $R$ is assumed to be large, the curve 
$\Gamma_{\delta, R}$ has the same starting point $p_0$ as the polygonal curve $\Gamma_+$. Then, as the curve parameter starts moving along the real axis, the corresponding point on the curve $\Gamma_{\delta, R}$ remains fixed until the parameter reaches the first indentation near the point $\beta_1$. Let us parametrize this circular indentation arc explicitly by $t(\theta)=\beta_1+\delta e^{i\theta}\!$, $0\le\theta\le\pi$. We then have
$$[\text{Arg}\circ\Psi]_0(t(\theta))-[\text{Arg}\circ\Psi]_0(t(0))=(b_{11}\theta, b_{12}\theta)+ r_\delta(\theta)\,,$$
with the remainder term equal to
\begin{equation}\label{beta1}
r_\delta(\theta)=\biggl[\text{arg}_0\prod_{k\ne1}\!\Bigl(1+\frac{b_{k2}b_{12}\,\delta e^{i\theta}}{B_{k1}}\Bigr)^{b_{k1}}\!\!,\ \
\text{arg}_0\prod_{k\ne1}\!\Bigl(1+\frac{b_{k2}b_{12}\,\delta e^{i\theta}}{B_{k1}}\Bigr)^{b_{k2}}\biggr].\end{equation}
Here $B_{k1}=\det(b_k,b_1)$ and $\text{arg}_0$ denotes the principal branch of the argument that vanishes for positive real numbers. It is clear that $r_\delta=O(\delta)$ and $r_\delta(0)=r_\delta(\pi)=0$, so this part of $\Gamma_{\delta,R}$ is a curve running very close to the first line segment in $\Gamma_+$, and having the same end points as this segment. In particular, our chosen point $p$ is not on this part of the curve $\Gamma_{\delta,R}$  and the argument changes along the two curve pieces seen from $p$ are exactly the same.
\smallskip

Continuing in this manner along the curve $C_{\delta, R}$, we may proceed similarly at all the small indentations to find that the image curve $\Gamma_{\delta, R}$ will closely follow along the consecutive edges of the polygonal curve $\Gamma_+$. Having thus arrived at the point $-R$ on the curve $C_{\delta, R}$, it just remains to follow the large half circle back to the starting point $R$. Assume first that $b_{N2}=0$, so that the last singularity $\beta_N$ is located at infinity. Parametrizing the large half circle by $t(\theta)=-Re^{-i\theta}$, for $0\le\theta\le\pi$, we then get
$$[\text{Arg}\circ\Psi]_0(t(\theta))-[\text{Arg}\circ\Psi]_0(t(0))=(b_{N1}\theta, b_{N2}\theta)+ s_R(\theta)\,,$$
with 
$$
s_R(\theta)=\biggl[\text{arg}_0\prod_{k\neq N}\!\Bigl(1\!+\!\frac{b_{k1}R^{-1}(1\!-\!e^{i\theta})}{b_{k2}-b_{k1}R^{-1}}\Bigr)^{b_{k1}\!\!\!},\ \
\text{arg}_0\prod_{k\neq N}\!\Bigl(1\!+\!\frac{b_{k1}R^{-1}(1\!-\!e^{i\theta})}{b_{k2}-b_{k1}R^{-1}}\Bigr)^{b_{k2}}\biggr].$$
Here we see that  $s_R(0)=s_R(\pi)=0$, and that $s_R=O(R^{-1})$ as $R\to\infty$. Again we conclude that the corresponding part of the image curve $\Gamma_{\delta,R}$
follows closely along the final segment of $\Gamma_+$, thus producing an equal argument change as seen from the fixed point $p$.
If on the other hand  $b_{N2}\ne0$, then moving along the half circle from $-R$ to $R$ just results in making a small loop at the vertex $p_0$, which gives no additional change in arguments. Hence we have shown that the curve $\Gamma_{\delta,R}$ follows the  curve $\Gamma_+$ arbitrarily closely, and that for any chosen point $p$ in the complement of 
$\Gamma_+$, the two curves both have the same winding number with respect to $p$. 
\smallskip

In order to relate the winding number of the image curve $\Gamma_{\delta,R}$ around the point $p$ to the number of preimages of $p$ in the upper half plane, we shall consider the Jacobian $J$ of the principal branch mapping. It is not hard to prove, see for instance \cite{M}, that the set of critical points for the mappings $\Log$ and $\Arg$ on a plane algebraic curve is equal to the inverse image $\gamma^{-1}(\mathbf{R}P_1)$, where $\gamma$ as before denotes the logarithmic Gauss mapping. Since our Horn--Kapranov parametrization $\Psi$ is inverse to $\gamma$, it follows that the Jacobian $J$ is nonvanishing at all non-real points $[1:t]\in\mathbf{C}P_1$. In fact, we claim that one has $J(t)<0$ for $\text{Im}\,t>0$ and vice versa. To see that the principal branch mapping $[\Arg\circ\Psi]_0$ is indeed orientation-reversing in the upper half plane, we observe that the curve 
$C_{\delta, R}$ is negatively oriented with respect to its interior, so we just have to prove that its image $\Gamma_{\delta,R}$ is instead a positively oriented boundary. This will follow if we show that the image of a point $t$ keeps to the left of the polygonal curve $\Gamma_+$ as $t$ moves along $C_{\delta, R}$. For this it suffices to show that 
$\det(b_1,r_\delta(\theta))>0$ for all $0<\theta<\pi$, with $r_\delta$ from (\ref{beta1}). Some straightforward computations now firstly show that 
$$r_\delta(\theta)\ =\ \biggl[\,\sum_{k>1}\frac{b_{k1}b_{k2}b_{12}\,\delta\sin\theta}{B_{k1}}\,,\ \sum_{k>1}\frac{b_{k2}^2b_{12}\,\delta\sin\theta}{B_{k1}}\biggr]\ +\ O(\delta^2)\,,$$
and secondly that the determinant $\det(b_1,r_\delta(\theta))$ therefore equals $b_{12}^2\,\delta\sin\theta+O(\delta^2)$. Since $\delta$ is small and $b_{12}\ne0$, this proves our claim. Observing that the Jacobian $J$ is an odd function with respect to $\text{Im}\,t$, we also see that $J(t)>0$ for $\text{Im}\,t<0$.
\smallskip

The fact that $J$ is non-zero with constant sign in the upper half plane allows us to count preimages of $p$ as desired. Indeed, we continuously shrink the curve $C_{\delta, R}$ slightly, without  letting its image curve hit the point $p$, so the new curve bounds a compact region in the open upper half plane. Then we subdivide this region into finitely many small patches on which our mapping is bijective. The negatively oriented boundaries of these patches then get mapped to positively oriented closed curves in $\mathbf{R}^2$ having winding numbers zero or one with respect to the point $p$, according to whether $p$ belongs to the image or not. This shows that the winding number of $\Gamma_+$ around $p$ indeed coincides with the number of preimages of $p$ in the upper half plane, and that we have $\partial\mathcal{A}_+'=\Gamma_+$. To study the image 
$\mathcal{A}_-'$ of the lower half plane we argue analogously, using instead the curve obtained by reflecting $C_{\delta, R}$ in the real axis. Notice that this new curve is positively oriented with respect to its interior. Since $J>0$ in the lower half plane, this means that $\Gamma_-$ is also positively oriented with respect to $\mathcal{A}_-'$.
\smallskip

In the above arguments we had made the assumption that there were no parallel vectors among the rows $b_j$. If there are some vectors, say $b_k, b_{k+1},\ldots,b_{k+\ell}$, that are parallel, then the arguments just need a very slight modification. The starting point $p_0$ of the cycle $\Gamma_+$ will still coincide with the image of the starting point $R$ of the contour $C_{\delta, R}$, but as we follow the $\delta$-indentation at the singularity $\beta_k$, which now coincides with $\beta_{k+1}=\ldots=\beta_{k+\ell}$, the image curve will run along the edge of $\Gamma_+$ corresponding to the summed vector $b_k+b_{k+1}+\ldots+b_{k+\ell}$. This means that the image curve $\Gamma_{\delta,R}$ will be istotopic to $\Gamma_+$ in $\mathbf{R}^2\setminus\{p\}$ just as before.
\end{proof}

{\it Remark.} With Theorem \ref{sats1} we have obtained a precise description of the principal coamoeba chain $\mathcal{A}_0'$ including information about multiply covered points. But what about the actual set-theoretic coamoeba? It differs from the (closed) underlying set of the coamoeba chain in that it does not contain all its boundary points. In fact, it should have become clear from the proof of Theorem \ref{sats1} that the real intervals of $\mathbf{R}P_1\setminus\{\beta_1,\ldots,\beta_N\}$ get mapped to the vertices of the polygonal curve $\Gamma_0=\Gamma_++\Gamma_-$. (See the examples in Figure~2, where  these vertices are clearly marked.) The edges of $\Gamma_0$ however, are not attained in the image of the principal branch mapping $[\Arg\circ\Psi]_0$, for they correspond to the singular points $\beta_k$. This means that, set-theoretically, the principal coamoeba is equal to the interior of the underlying set of the coamoeba chain, but with the vertices of $\Gamma_0$ added to it.
\smallskip

\begin{figure}[H]\label{fig2}
\begin{center}
\vskip-.2cm
\includegraphics[width=10cm]{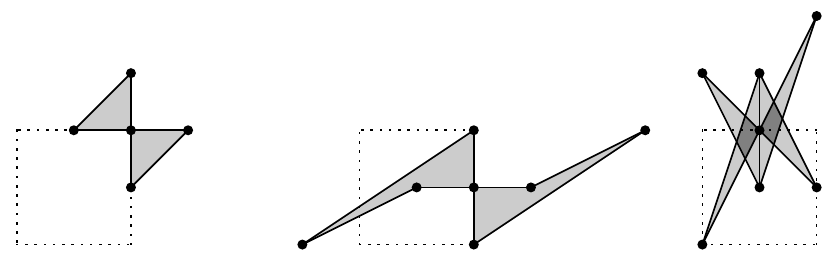}
\end{center}
\vskip.2cm
\caption{\it The principal coamoeba chains $\mathcal{A}_0'$ for the running examples. Observe the darker doubly covered areas in the third case.}
\end{figure}

\section{Associated zonotopes and area formulas}

\noindent Having obtained a clear picture of the principal coamoeba chain $\mathcal{A}_0'\subset\mathbf{R}^2$, whose image under the canonical projection to the torus $\mathbf{T}^2$
is equal to the true coamoeba chain $\mathcal{A}_B'$, we now introduce another, much simpler polygonal chain $\Pi_0$, which is also defined in terms of the given matrix $B$. In fact, $\Pi_0$ is nothing but the closed convex polygon obtained as the Minkowski sum 
$$ \Pi_0=S_1+S_2+\ldots+S_N$$
of the line segments $S_k=\mathrm{conv}\bigl\{(0,0),\pi b_k\bigr\}$, where the $b_k$ as before denote the row vectors of $B$. Being a Minkowski sum of line segments, the polygon $\Pi_0$ is by definition a so-called \emph{zonotope}. As all zonotopes it is symmetric with respect to its center, and the fact that the vectors $b_k$ sum to zero implies that the center of $\Pi_0$ is located at the origin. 
The positively oriented boundary of $\Pi_B$ can be obtained by placing the vectors $\pm\pi b_k$ cyclically after one another in counterclockwise order.

\begin{figure}[H]\label{fig3}
\begin{center}
\vskip0cm
\hskip.8cm\includegraphics[width=10.4cm]{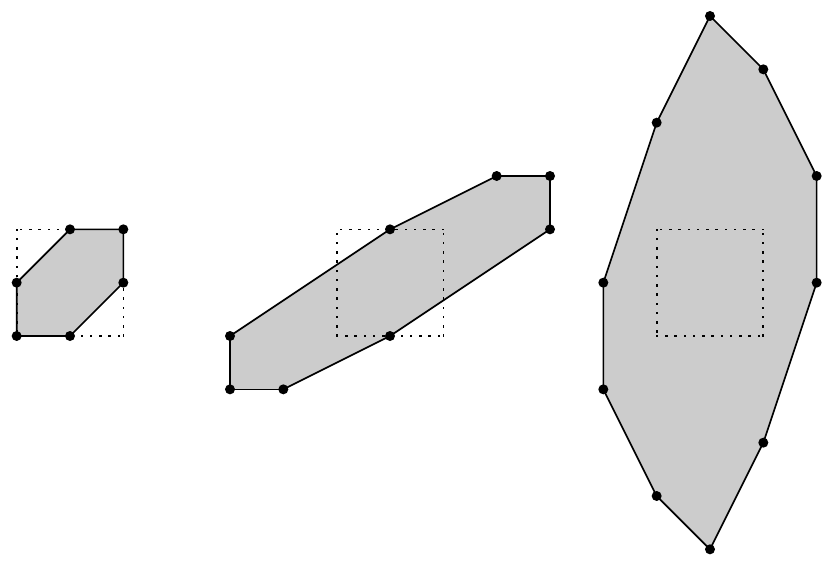}
\end{center}
\vskip0cm
\caption{\it The zonotopes $\Pi_0$ for our three running examples.}
\end{figure}

\noindent A dual representation of the zonotope $\Pi_0$ in terms of linear inequalities is given by
$$\Pi_0=\Bigl\{\,\theta\in\mathbf{R}^2\,;\ \big|\langle\theta,\xi\rangle\big|\le\frac{\pi}{2}\sum_{k=1}^N\big|\langle b_k,\xi\rangle\big|\,,\  \text{for all}\ \xi\in\mathbf{R}^2\,\Bigr\}\,.$$
Observe that there are in fact only finitely many inequalities involved here, one for each normal direction $\xi=(b_{k2},-b_{k1})$.
\smallskip

Our next objective is to compute the area both of the zonotope $\Pi_0$ and of the principal coamoeba chain $\mathcal{A}_0'$. The area of $\Pi_0$ poses no problem, 
because there is a well known formula for the volume of a general zonotope, see for instance \cite{S}, which in our setting takes the form
\begin{equation}\label{areazonotope}\Area\,(\Pi_0)=\pi^2\sum_{j<k}|\det(b_j,b_k)|\,.\end{equation} 
\vskip-.1cm \noindent Since we are taking absolute values in the above sum, the ordering of the vectors is immaterial here. 
\smallskip

We then want to derive a similar formula for the area of the principal coamoeba chain $\mathcal{A}_0'$. As we have seen, the principal coamoeba can overlap itself, so we really should take the multiplicities into account when we calculate its area.  Let us define precisely what we mean by the area of a (simplicial) chain.

\begin{dfn}
{\rm The \emph{area} of a simplicial $2$-chain $\sum_{j\in J}m_j\,\sigma_j$ is defined to be $\sum_{j\in J}m_j\Area\,(\sigma_j)$, where $m_j$ is the multiplicity of the component 
$\sigma_j$.}
\end{dfn}

This definition allows also for negative multiplicities, and such negatively signed areas will actually appear in our proof of the following result. Recall that we have ordered the row vectors $b_j$ of the matrix $B$ in a clockwise projective manner, that is, so that the slope of each vector is smaller than that of the preceding one.

\begin{pro}\label{areacoam} 
Let the vectors $b_j$ be ordered clockwise projectively. The area of the principal coamoeba chain is then given by 
\begin{equation}\label{areacoamoeba}\Area\,(\mathcal{A}'_0)=\pi^2\sum_{j<k}\det(b_j,b_k)\,.
\end{equation}
\end{pro}
 
\begin{proof}
Given any vectors $v_1,v_2,\ldots,v_r$ in $\mathbf{R}^2$, we introduce the notation $[v_1,v_2,\ldots,v_r]$ for the oriented polygonal curve obtained by placing the vectors one after another, starting by putting $v_1$ at the origin. In this notation we have 
$$\Gamma^\pm=p_0\pm[\pi b_1,\pi b_2,\ldots,\pi b_N]\,,$$
 except that if there are some parallel vectors with opposite signs among the $b_j$, we get on the right hand side some harmless extra edges that are traced in both directions.
Making a translation and re-scaling by $1/\pi$, we thus see that the area of the principal coamoeba chain $\mathcal{A}_0'$ that we wish to find, is equal to $2\pi^2$ times the area of the chain bounded by the $1$-cycle $[b_1,b_2,\ldots,b_N]$. We shall now compute this area by summing the signed area of certain oriented triangles.
\smallskip

Let us write $\hat b_j$ for the vector $b_j+b_{j+1}+\ldots+b_N$, and consider the $1$-cycle
$$[-\hat b_2, b_2,\hat b_3,-\hat b_3, b_3,\hat b_4,-\hat b_4, b_4,\hat b_5,\ldots, -\hat b_{N-1}, b_{N-1},\hat b_N]\,.$$
Since $b_1=-\hat b_2$ and $b_N=\hat b_N$, we see that this new cycle is just the extension of the cycle  $[b_1,b_2,\ldots,b_N]$ obtained by inserting the two extra 
edges $\hat b_k$ and $-\hat b_k$ between the original edges $b_{k-1}$ and $b_k$, for $k=3,\ldots, N-1$.
\smallskip

Obviously, the winding number with respect to each point in the complement of this curve is the same as for the original cycle $[b_1, b_2,\ldots, b_N]$. 
Notice that each triplet $[-\hat b_k, b_k,\hat b_{k+1}]$ can be written $[b_1+\ldots+b_{k-1},b_k,b_{k+1}+\ldots+b_N]$, and since the total sum $b_1+b_2+\ldots+b_N$ is equal to zero, the triplet is a cycle that bounds an oriented triangle $T_k$. The signed area of such a triangle $T_k$ is given by $\det(-\hat b_k,b_k)/2$, and we conclude that 
$$\text{Area}\,(\mathcal{A}_0')=2\pi^2\sum_{k=2}^{N-1}\text{Area}\,(T_k)=\pi^2\sum_{k=2}^N\det(b_1+\ldots+b_{k-1},b_k)=\pi^2\sum_{j<k}\det(b_j,b_k),$$
where we have used the fact that $\det(b_1+\ldots+b_{N-1},b_N)=\det(-b_N,b_N)=0$, and the last equality follows from the bilinearity of the determinant.
\smallskip

Notice also that the expression $\sum_{j<k}\det(b_j,b_k)$ is invariant under cyclic permutations. Indeed, it is enough to check that the sum does not change if we take the vectors in the new order $b_N,b_1,b_2,\ldots,b_{N-1}$, and then it is only the contribution from determinants containing $b_N$ that matters. But this contribution is zero for both the original and the new ordering, since we have
$$-\sum_{j=1}^{N-1}\det(b_N,b_j)=\sum_{j=1}^{N-1}\det(b_j,b_N)=\det(b_1+\ldots+b_N,b_N)=\det(0,b_N)=0\,.$$
\end{proof}
\bigskip

\section{The coamoeba and its complement}

\noindent We are going to project both the principal coamoeba chain $\mathcal{A}_0'$ and the zonotope $\Pi_0$ to the torus $\mathbf{T}^2=(\mathbf{R}/2\pi\mathbf{Z})^2$ by means of the canonical projection $\mathbf{R}^2\to\mathbf{T}^2$. The image of $\mathcal{A}_0'$ will be the coamoeba chain $\mathcal{A}_B'$, which is our main object of study, and the image of $\Pi_0$ will be an easily described chain that we denote by $\Pi_B$.
\smallskip

Before actually performing the projection, let us first determine the mutual location of the principal coamoeba $\mathcal{A}_0'$ and the zonotope $\Pi_0$. We have already located the common starting point $p_0$ for the two boundary cycles $\Gamma_\pm$ of $\mathcal{A}_0'$. Let us now also fix a starting point $q_0$ on the boundary of the zonotope 
$\Pi_0$. We let this $q_0$ be the ``southwesternmost" vertex of $\Pi_0$, that is, the sum of all vectors $\pi b_j$ with strictly negative second coordinate plus the sum of those $\pi b_j$ that have zero second coordinate and negative first coordinate. The following simple result shows that the points $p_0$ and $q_0$ get projected to the same point on the torus $\mathbf{T}^2$.
\bigskip

{\smc Lemma.} {\it
Let $p_0$ and $q_0$ denote the initial vertices of $\mathcal{A}_0'$ and $\Pi_0$ as explained above. Then one has $p_0\equiv q_0\equiv -q_0\ (\text{\rm mod}\ 2\pi\mathbf{Z}^2\,)$.}

\begin{proof} Notice first that, since $q_0$ is a point in the denser lattice $\pi\mathbf{Z}^2$, it is clear that the difference $q_0-(-q_0)=2q_0$ belongs to $2\pi\mathbf{Z}^2$, and hence that $q_0\equiv -q_0$. Now, from the way in which the vertex $q_0$ was chosen, it follows that its coordinates are given by
$$q_{01}\ =\ \pi\!\!\sum_{b_{j2}<0}b_{j1}\ +\ \pi\!\!\sum_{\substack{b_{j2}=0\\ b_{j1}<0}}b_{j1}\quad \text{and} \quad q_{02}\ =\ \pi\!\!\sum_{b_{j2}<0}b_{j2}\,.$$
The vertex $p_0$ on the other hand, was defined as the value of the principal branch mapping $[\Arg\circ\Psi]_0$ at a large positive value of $t$.
Recalling the form (\ref{horn}) of the Horn--Kapranov mapping $\Psi$, we see that, modulo the lattice $2\pi\mathbf{Z}^2$, the coordinates of $p_0$ are given by 
$$\phantom{-}p_{01}\ =\ \arg\bigl(\prod_{b_{j2}<0} b_{j2}^{b_{j1}}\ \times \prod_{\substack{b_{j2}=0\\ b_{j1}<0}} b_{j1}^{b_{j1}}\bigr)\quad \text{and} \quad p_{02}\ =\ \arg\bigl(\prod_{b_{j2}<0} b_{j2}^{b_{j2}}\bigr)\,.$$
Since the argument of a negative number is equal to $\pi$ modulo $2\pi$, we see that the above coordinate expressions for $q_0$ and $p_0$ do indeed coincide modulo the lattice.
\end{proof}
\smallskip

Consider now the natural projection $\mathbf{R}^2\to\mathbf{T}^2$, and the images $\mathcal{A}_B'$ and $\Pi_B$ of the chains $\mathcal{A}_0'$ and $\Pi_0$ respectively.
Let us also introduce the notation $$\det\!^{+}(b_j,b_k)=\mathrm{max}\{\det(b_j,b_k),0)\}.$$
\begin{figure}[H]\label{fig4}
\begin{center}
\vskip0cm
\includegraphics[width=12cm]{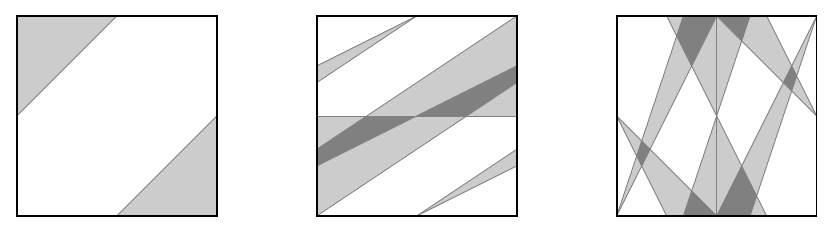}
\end{center}
\vskip0cm
\caption{\it The coamoeba chains $\mathcal{A}_B'$ for our recurring examples drawn in the torus $\mathbf{T}^2\!$. The lightly shaded regions are covered once and the darker areas are doubly covered.}
\end{figure}
\vskip-.4cm

Our next result reveals a close complementary relationship between the coamoeba and the zonotope.
\smallskip

\begin{thm}\label{sats2} The summed chain $\mathcal{A}'_B+\Pi_B$ is a cycle, and hence equal to $m_B\mathbf{T}^2$, with some integer multiplicity $m_B$. In fact, provided that the vectors $b_k$ are ordered clockwise projectively, this multiplicity is given by the formula
$$m_B=\frac{1}{2}\sum_{j<k}\det\!^+(b_j,b_k)\,.$$
\end{thm}

\begin{proof}
In order to prove that $\mathcal{A}'_B+\Pi_B$ is a cycle we are going to show that the oriented boundary segments of $\mathcal{A}_0'$ and $\Pi_0$ cancel each other out pairwise when mapped to the torus. From the above lemma we know that the starting points $p_0$ and $\pm q_0$ get projected onto the same point in the torus. Let us now follow the respective boundary cycles segment by segment, and verify that there is a cancellation as claimed.
\smallskip

To begin with, we treat the case with no parallel vectors among the $b_k$. The boundary cycle $\Gamma_0$ then has an oriented edge given by the vector $\pi b_1$ going \emph{from} the vertex $p_0$ to the neighboring vertex $p_0+\pi b_1$, and another oriented edge given by the vector $-\pi b_1$ going from $p_0$ to $p_0-\pi b_1$. On the zonotope on the other hand, we have an oriented boundary edge given by one of the two vectors 
$\pm\pi b_1$ but now \emph{arriving at} the vertex $q_0$ from the neighboring vertex $q_0\mp\pi b_1$, and another oriented edge given by the opposite vector $\mp\pi b_1$ arriving at $-q_0$ from $-q_0\pm\pi b_1$. From this description it is clear that each of the two coamoeba edges, when mapped to the torus, exactly coincides with one reversely oriented edge from the zonotope.
\smallskip

Now we note that the four neighboring vertices $p_0\pm\pi b_1$ and $\pm(q_0\pm\pi b_1)$ mentioned above are again mapped to one single point on the torus, so we can repeat the same reasoning, starting at these points and cancelling oriented edges given by $\pm\pi b_2$. Continuing in the same manner we thus establish that the boundaries of the projected chains $\mathcal{A}_B'$ and $\Pi_B$ indeed sum up to zero.
\smallskip

In the presence of parallel vectors, the argument needs to be adapted somewhat. Suppose for instance that the vectors $b_k, b_{k+1},\ldots,b_{k+\ell}$ are parallel, but the next vector $b_{k+\ell+1}$ has a different slope. This implies that we can write $b_{k+j}=r_jb_k$, for some rational numbers $r_j$. Let us introduce the notation $b_k'$ for the summed vector $b_k+b_{k+1}+\ldots+b_{k+\ell}=(1+r_1+\ldots+r_\ell)b_k$, and $b_k''$ for the possibly longer vector $(1+|r_1|+\ldots+|r_\ell|)b_k$. As before there will be four vertex points, two on the coamoeba and two on the zonotope, that all correspond to the same torus point, but this time with the property that the two coamoeba vertices have the vectors $\pi b_k'$ and $-\pi b_k'$ as oriented edges emanating from them, while the two zonotope vertices have the vectors $\pi b_k''$ and $-\pi b_k''$ terminating at them.
\smallskip

Unless all the parallel vectors $b_k, b_{k+1},\ldots,b_{k+\ell}$ have the same direction, so that all the $r_j$ are positive, the edges of the coamoeba and the reversely oriented, but longer, edges of the zonotope will no more cancel each other out completely. However, the two remaining edge pieces from the zonotope boundary will be of the form 
$\pm\pi(b_k''-b_k')=\pm2\pi(r_1^++\ldots+r_\ell^+)b_k$, where $r_j^+=\max(r_j,0)$, and this means that on the torus they both represent the same \emph{closed} curve, but with opposite orientations. They therefore cancel each other out, and again we conclude that the boundary of $\mathcal{A}_B'+\Pi_B$ vanishes.
\smallskip

Now, the only $2$-cycles on the torus $\mathbf{T}^2$ are integral multiples of $\mathbf{T}^2$ itself, and since our chains are positively oriented we must therefore have 
$\mathcal{A}_B'+\Pi_B=m_B\mathbf{T}^2$, for some positive integer $m_B$. In order to find this multiplicity $m_B$ we note that the area (with multiplicities) is preserved under the projection $\mathbf{R}^2\to(\mathbf{R}/2\pi\mathbf{Z})^2.$ This means that we can compute $m_B$ as the area of $\mathcal{A}_B'+\Pi_B$ divided by the area $4\pi^2$ of the torus. Using the area formulas (\ref{areazonotope}) and (\ref{areacoamoeba}) for the zonotope $\Pi_0$ and for the principal coamoeba chain $\mathcal{A}_0'$ respectively, we obtain
$$m_B=[\pi^2\sum_{j<k}\det(b_j,b_k)+\pi^2\sum_{j<k}|\det(b_j,b_k)|]/4\pi^2
=\frac{1}{2}\sum_{j<k}\det\!^+(b_j,b_k)\,.$$
\end{proof}

{\smc Examples.} ({\it i$\,$})\  In our first example case the appropriately re-ordered vectors are $b_1=(0,1)$, $b_2=(-1,-1)$, and $b_3=(1,0)$, and hence the area computations  become
$$
\pi^{-2}\text{Area}\,(\mathcal{A}'_B)=
\left|\begin{matrix}
\,\,\,\,0 & 1\\
-1 & \!\!\!\!-1\end{matrix}\,\right|+
\left|\,\begin{matrix}
0 & 1\\
1 & 0
\end{matrix}\,\right|+
\left|\begin{matrix}
-1 & \!\!\!\!-1\\
\,\,\,\,1 & 0
\end{matrix}\,\right|= 1-1+1=1\,;\phantom{---}
$$

$$\pi^{-2}\text{Area}\,(\Pi_B)=
\Big|\left|\begin{matrix}
\,\,\,\,0 & 1\\
-1 & \!\!\!\!-1
\end{matrix}\,\right|\Big|+
\Big|\left|\,\begin{matrix}
0 & 1\\
1 & 0
\end{matrix}\,\right|\Big|+
\Big|\left|\begin{matrix}
-1 & \!\!\!\!-1\\
\,\,\,\,1 & 0
\end{matrix}\,\right|\Big|
= 1+1+1=3\,.
$$
\smallskip

\noindent We see in Figures 3 and 4 that the combined chain $\mathcal{A}'_B+\Pi_B$ precisely covers the torus $\mathbf{T}^2$, so that in this case $m_B=1$, which agrees with the formula from Theorem~\ref{sats2}:
$$m_B=\frac{1}{2}\sum_{j<k}\det{\!}^+(b_j,b_k)=\frac{1}{2}\big(1+0+1\big)=1\,.$$
\smallskip

({\it ii$\,$})\  The second example consists of the vectors $b_1=(0,1)$, $b_2=(-3,-2)$, $b_3=(2,1)$, and $b_4=(1,0)$. This time we obtain the areas 

$$
\pi^{-2}\text{Area}\,(\mathcal{A}'_B)=
\left|\begin{matrix}
\,\,\,\,0 & 1\\
-3 & \!\!\!\!-2\end{matrix}\,\right|+
\left|\,\begin{matrix}
0 & 1\\
2 & 1
\end{matrix}\,\right|+
\left|\,\begin{matrix}
0 & 1\\
1 & 0
\end{matrix}\,\right|+
\left|\begin{matrix}
-3 & \!\!\!\!-2\\
\,\,\,\,2 & 1
\end{matrix}\,\right|+
\left|\begin{matrix}
-3 & \!\!\!\!-2\\
\,\,\,\,1 & 0
\end{matrix}\,\right|+
\left|\,\begin{matrix}
2 & 1\\
1 & 0
\end{matrix}\,\right|=
$$
$$= 3-2-1+1+2-1=2\,;\phantom{---}\pi^{-2}\text{Area}\,(\Pi_B)=
3+2+1+1+2+1=10\,.
$$
The area of $\mathcal{A}'_B+\Pi_B$ is thus equal to $12\pi^2$ and dividing by the area of the torus $\mathbf{T}^2$, we get the multiplicity $m_B=12\pi^2/4\pi^2=3$.

\bigskip\bigskip

\smallskip

({\it iii$\,$})\  In our third and final case we have $b_1=(1,-1)$, $b_2=(-1,2)$, $b_3=(0,-2)$, $b_4=(1,3)$, and $b_5=(-1,-2)$. Computing the areas we now get

$$\pi^{-2}\text{Area}\,(\mathcal{A}'_B)=
\left|\,\begin{matrix}
\,\,\,\,1 & \!\!-1\\
-1 & \,\,2
\end{matrix}\,\right|+
\left|\begin{matrix}
\,1 & \!-1\\
\,0 & \!-2
\end{matrix}\,\right|+
\left|\begin{matrix}
\,1 &  \!\!-1\\
\,1 & \,\,3
\end{matrix}\,\right|+
\left|\,\begin{matrix}
\,\,\,\,1 & \!-1\\
-1 & \!-2
\end{matrix}\,\right|+
\left|\begin{matrix}
-1 & \,\,2\\
\,\,\,\,0 & \!\!-2
\end{matrix}\,\right|+$$
\smallskip
$$\phantom{-------}\!\!\!\!+
\left|\begin{matrix}
-1 & \!2\\
\,\,\,\,1 &\!3
\end{matrix}\,\right|+
\left|\,\begin{matrix}
-1 & \,\,2\\
-1 & \!\!-2
\end{matrix}\,\right|+
\left|\begin{matrix}
\,0 & \!\!-2\\
\,1 & \,\,3
\end{matrix}\,\right|+
\left|\begin{matrix}
\,\,\,\,0 & \!\!-2\\
-1 & \!\!-2
\end{matrix}\,\right|+
\left|\,\begin{matrix}
\,\,\,\,1 & \,\,3\\
-1 & \!\!-2
\end{matrix}\,\right|=
$$\smallskip
$$\!\!\!\!=1-2+4-3+2-5+4+2-2+1=2\,;$$
$$
\,\pi^{-2}\text{Area}\,(\Pi_B)=1+2+4+3+2+5+4+2+2+1=26\,.\phantom{-------}
$$
From Theorem \ref{sats2} we then find the multiplicity $m_B=(1+4+2+4+2+1)/2=7$.
\bigskip

\section{Significance of the multiplicity}

\noindent We shall now relate the multiplicity $m_B$ of the summed chain $\mathcal{A}'_B+\Pi_B$ to the original point configuration $A\subset\mathbf{Z}^n$, $n=N-3$, from which the matrix  $B$ was obtained as a Gale dual. 
\begin{thm}\label{mult} The multiplicity $m_B$ from Theorem~\ref{sats2} satisfies $$m_B=d_B\,, $$where $d_B$ is the normalized volume $n!\Vol(\mathrm{conv}(A))$ of the convex hull of the point configuration $A\subset\mathbf{Z}^n$.
\end{thm}

\begin{proof}
Remember that we have ordered the vectors $b_k$ according to decreasing slope, with the normal of $b_1$ having the largest finite slope.
Let $L$ be an almost horizontal line through the origin in $\mathbf{R}^2$ with slope $-\varepsilon$, so that its normal slope $\varepsilon^{-1}$ is larger than $\beta_1=-b_{11}/b_{12}$. Also let $L^+$ and $L^-$ denote the two rays $L\cap\{x\ge0\}$ and $L\cap\{x\le0\}$ respectively. By Gale duality, see for instance \cite{DS}, one knows that
$$d_B\ \ =\!\!\!\!\sum_{\mathrm{Cone}(b_j,b_k)\supset L^+}\!\!\!\!|\det(b_j,b_k)|\ \ \ =\!\!\!\!\sum_{\mathrm{Cone}(b_j,b_k)\supset L^-}\!\!\!\!|\det(b_j,b_k)|\,,$$
where the summation is over all pairs $j<k$ such that the convex cone spanned by the two vectors $b_j$ and $b_k$ contains the ray $L^+$ or $L^-$.   
In order to prove the theorem it will therefore suffice to show that 
\begin{equation}\label{gale}\phantom{--}\sum_{j<k}\det\!^+(b_j,b_k)\ \ \ =\!\!\!\!\sum_{\mathrm{Cone}(b_j,b_k)\supset L^+}\!\!\!\!|\det(b_j,b_k)|\ \ \ +\!\!\!\!\sum_{\mathrm{Cone}(b_j,b_k)\supset L^-}\!\!\!\!|\det(b_j,b_k)|\,.\end{equation}

Notice first that if $b_j$ and $b_k$ lie on the same side of $L$, then because of the special clockwise ordering, the determinant $\det(b_j,b_k)$ is negative (or zero), and the cone 
$\mathrm{Cone}(b_j,b_k)$ does not contain either of the rays $L^+$ or $L^-$. Such pairs $j,k$ therefore give zero contribution to both sides of (\ref{gale}). Suppose next that $b_j$ is below $L$ and $b_k$ is above $L$. Then, since $j<k$, the positive angle from $b_j$ to $b_k$ is $\le\pi$ so the determinant is $\ge0$, and unless the two vectors are parallel, the convex cone $\mathrm{Cone}(b_j,b_k)$ contains the ray $L^+$. We thus get the same contribution $|\det(b_j,b_k)|$ to both sides of (\ref{gale}). The remaining case is similar.
\end{proof}
\bigskip

{\smc Examples.} In our example cases ({\it ii$\,$}) and ({\it iii$\,$}) the Gale duals $A$ can be given by the respective matrices 
$$ A=\left(\,
\begin{matrix} 1 & 1 & 1 & 1 \\ 1 & 2 & 3 &  0 
\end {matrix}\,\right)\quad\text{and}\quad A=\left(\,
\begin{matrix} 1 & 1 & 1 & 1 & 1 \\ 2 & 1 & 0 & 2 & 3 \\ 3 & 2 & 0 & 1 & 2
\end {matrix}\,\right),$$

\noindent which means that the corresponding point configurations are 
$$\{\,0,1,2,3\,\}\subset\mathbf{Z}\quad \text{and} \quad \{\,(0,0),(2,1),(1,2),(3,2),(2,3)\,\}\subset\mathbf{Z}^2\,.$$ 
The content of the identity $m_B=d_B$ from Theorem \ref{mult} in these examples is that the multiplicity $m_B=3$ from case ({\it ii$\,$}) concides with the length of the segment 
$[0,3]$, while $m_B=7$ from  case ({\it iii$\,$}) is equal to twice the area of the convex hull shown on the right in Figure~5.

\vskip-1cm
\begin{figure}[H]\label{fig5}
\begin{center}
\includegraphics[width=10cm]{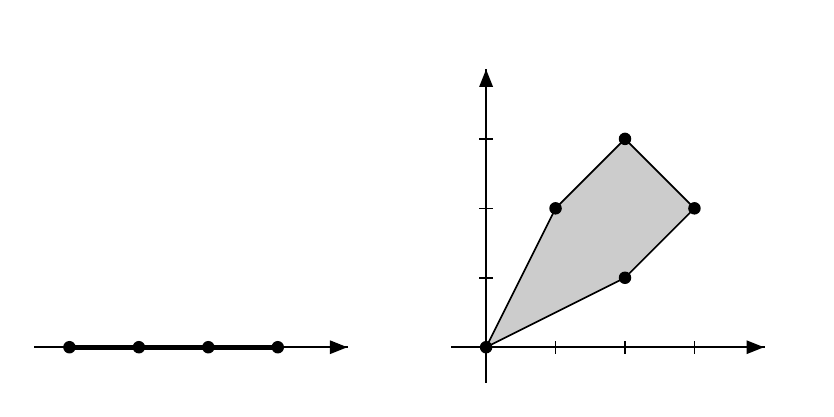}
\end{center}
\caption{\it The convex hull of the configuration $\{0,1,2,3\}$ (left), and of the configuration $\{(0,0),(2,1),(1,2),(3,2),(2,3)\}$ (right).}
\end{figure}

{\it Remark.} Recall that associated to the configuration $A$ there is the toric variety 
$$X_A=\overline{\{(z^{\alpha_1}:z^{\alpha_2}:\ldots:z^{\alpha_N})\,;\, z\in\mathbf{C}_*^n\}\!}\,\,\subset\,\mathbf{C}P_{N-1}\,.$$ 
It follows from Kouchnirenko's theorem that  $d_B$ also equals the degree of $X_A$, see for instance \cite{GKZ}. In the generic case when the projectively dual variety of $X_A$ is a hypersurface, it is in fact equal to the $A$-discriminantal hypersurface $D_A(a)=0$.

\end{document}